\documentclass[11pt]{amsart}
\usepackage{amsmath}
\usepackage{amssymb, verbatim}
\usepackage{pstricks,pst-node,pst-plot}
\usepackage{comment}
\usepackage{color}
\usepackage{extarrows}

\title[]{The Strominger system in the square of a K\"ahler class}
\author[T. C. Collins]{Tristan C. Collins}
  \email{tristanc@mit.edu}
  \address{Department of Mathematics, Massachusetts Institute of Technology, 77 Massachusetts Avenue, Cambridge, MA 02139}
 \thanks{T.C.C is supported in part by NSF CAREER grant DMS-1944952 and an Alfred P. Sloan Fellowship. }

 \author[S. Picard]{Sebastien Picard}
  \email{spicard@math.ubc.ca}
  \address{Department of Mathematics, UBC, 1984 Mathematics Road,
    Vancouver, BC, Canada}
  \thanks{}
  
 \author[S.-T. Yau]{Shing-Tung Yau}
  \email{styau@tsinghua.edu.cn}
  \address{Department of Mathematics, Tsinghua University, Haidian District, Beijing 100084, China}
  \thanks{} 

\theoremstyle{plain}
\newtheorem{thm}{Theorem}[section]
\newtheorem{prop}[thm]{Proposition}

\newtheorem{lem}[thm]{Lemma}

\theoremstyle{definition}

\newtheorem{rk}[thm]{Remark}

\numberwithin{equation}{section}

\newcommand{\del}{\partial}

\newcommand{\dbar}{\overline{\del}}
\newcommand{\ddb}{i \del\dbar}

\newcommand{\F}{\mathcal{F}}

\newcommand{\be}{\begin{equation}}
\newcommand{\bea}{\begin{eqnarray}}
\newcommand{\eea}{\end{eqnarray}} 
\newcommand{\ee}{\end{equation}}

\renewcommand{\geq}{\geqslant}

\renewcommand{\epsilon}{\varepsilon}
\renewcommand{\phi}{\varphi}

\begin{document}

\dedicatory{Dedicated to the memory of Jean-Pierre Demailly}

\begin{abstract}
We study the Strominger system with fixed balanced class. We show that classes which are the square of a K\"ahler metric admit solutions to the system for vector bundles satisfying the necessary conditions. Solutions are constructed by deforming a Calabi-Yau metric and a Hermitian-Yang-Mills metric along a path inside the given cohomology class.
  \end{abstract}

\maketitle

\section{Introduction}
Let $M$ be a compact complex manifold of dimension $n$, and suppose $\tau \in H^{n-1,n-1}_{BC}(M,\mathbb{R})$ is a Bott-Chern cohomology class such that there exists a positive real $(1,1)$-form $\omega>0$ with
\[
  \omega^{n-1} \in \tau.
\]
A class $\tau$ with this property is called a balanced class. Here the Bott-Chern class is defined as
\[
H^{n-1,n-1}_{BC}(M,\mathbb{R}) = \{ \beta \in \Omega^{n-1,n-1}(X,\mathbb{R}) : d \beta = 0 \} / {\rm Im} \, \ddb.
\]
The cone of balanced classes appears in the study of movable curves initiated by Boucksom-Demailly-Paun-Peternell \cite{BDPP}; see \cite{FuXiao,Toma}. Here we consider the problem of finding optimal representatives in a given balanced class. Various approaches to this question have been proposed in the literature \cite{FPPZ,FWW,FWW15,GF-survey,Popo,PPZ18a,STW,TW,TsengYau}, and the Strominger system \cite{Strom} stands out as a natural candidate due to its importance in theoretical physics \cite{AGS,AOMSS,BBFTY,CdlOmcO,dlOSvanes,McOSvanes}.

In this note, we will solve the Strominger system in balanced classes which come from K\"ahler classes. Heterotic string theory on K\"ahler manifolds starts with the work of Candelas-Horowitz-Strominger-Witten \cite{CHSW}. Solutions with non-zero torsion from general gauge bundles over K\"ahler manifolds were obtained by Li-Yau \cite{LY05}, and Andreas and Garcia-Fernandez \cite{AGF,AGF2}. The difference with the present work is that we fix the cohomology class of the solution. The question of existence of solutions to the system on a non-K\"ahler manifold is widely open \cite{Fei,FeiYau,FHP,FIUV,FGV,FTY,FY,GoldsteinPro,LY86,OUV,PPZ18}.

Let $X$ be be a compact K\"ahler Calabi-Yau manifold of complex dimension $n=3$ equipped with a nowhere vanishing holomorphic $(3,0)$ form $\Omega$. Let $\tau \in H^{2,2}_{BC}(X,\mathbb{R})$ be such that
\[
\tau = [\omega_0^2]
\]
for a K\"ahler metric $\omega_0$. Let $E \rightarrow X$ be a holomorphic vector bundle such that
\[
c_2(E)=c_2(X), \quad {\rm deg}_{\omega_0}(E)=0
\]
and assume that $E$ is stable with respect to $[\omega_0]$. By Yau's theorem \cite{Yau78}, there exists a K\"ahler Ricci-flat metric $\hat{\omega} \in [\omega_0]$, and we may write $\tau = [\hat{\omega}^2]$. By the Donaldson-Uhlenbeck-Yau theorem \cite{Donaldson,UY}, there is a Hermitian-Yang-Mills metric $\hat{H}$, solving
\[
\hat{\omega}^2 \wedge F_{\hat{H}} = 0,
\]
where $F_{\hat{H}}$ is the curvature of the Chern connection of $\hat{H}$. We will seek a new metric $H$ on $E$ and a new metric $\omega_\Theta$ on $X$ of the form
\[
| \Omega |_{\omega_\Theta} \omega_\Theta^2 = |\Omega|_{\hat{\omega}} \, \hat{\omega}^2 + \Theta > 0, \quad \Theta = \ddb \beta
\]
for $\beta \in \Omega^{1,1}(X,\mathbb{R})$, so that
\[
\omega_\Theta^2 \wedge F_{H} = 0,
\]
\[
\ddb \omega_\Theta = {\rm Tr} \, R_\Theta \wedge R_\Theta - {\rm Tr} \, F_{H} \wedge F_{H}.
\]
Here $R_\Theta$ denotes the curvature of the Chern connection of $\omega_\Theta$. By construction, the metric $\omega_\Theta$ is conformally balanced; this means
\[
d ( |\Omega|_{\omega_\Theta} \omega_\Theta^2)=0.
\]
Furthermore, $\omega_\Theta$ has (conformally) balanced class
\[
  | \Omega |_{\omega_\Theta} \omega_\Theta^2 \in [ |\Omega|_{\hat{\omega}} \, \hat{\omega}^2] \in H^{2,2}_{BC}(X,\mathbb{R}).
\]
Since $|\Omega|_{\hat{\omega}}$ is constant for a Calabi-Yau metric $\hat{\omega}$, we are looking for solutions to the Strominger system in a rescaling of the balanced class $\tau$. Our main theorem states that the system is solvable in the large radius limit.

\begin{thm} Let $X$ be a complex manifold of dimension 3 equipped with a holomorphic volume form $\Omega$ and a K\"ahler Ricci-flat metric $\hat{\omega}$. Let $E \rightarrow X$ be a holomorphic vector bundle such that $c_2(E)=c_2(X)$ and suppose $E$ is of degree zero and stable with respect to $[\hat{\omega}^2]$. 
\smallskip  
\par There exists $L_0 \geq 1$ depending on $(X,E,\hat{\omega})$ with the following property. For all $M_0 \geq L_0$, there exists a pair $(\omega,H)$ with
\[
  [| \Omega|_\omega \, \omega^2] = M_0 [\hat{\omega}^2] \in H^{2,2}_{BC}(X,\mathbb{R}),
\]
which solves the system
\[
\omega^2 \wedge F_{H} = 0, \quad d (| \Omega|_\omega \, \omega^2)=0,
\]
\[
\ddb \omega = {\rm Tr} \, R_\omega \wedge R_\omega - {\rm Tr} \, F_{H} \wedge F_{H}.
\]
\end{thm}

\begin{rk} The system was previously solved by Andreas and Garcia-Fernandez \cite{AGF} for holomorphic bundles $E \rightarrow (X,\hat{\omega})$ over K\"ahler Calabi-Yau manifolds assuming the necessary conditions that $E$ is stable, degree 0, and with $c_2(E)=c_2(X)$. The setup in \cite{AGF,AGF2,LY05} is a bit different, as it does not fix the balanced class of the solution, though it was expected \cite{GF-survey} that this approach could be modified to control the balanced class. We confirm this here and present a method which fixes the balanced class, and we refer to \S \ref{section:tangentinstantons} for further discussion and comparisons to prior work.
\end{rk}

There is another way to view this result by introducing a parameter $\alpha'>0$ and looking for solutions in the fixed balanced class $\tau = [\hat{\omega}^2]$. From a solution produced by the theorem above with class $| \Omega|_\omega \, \omega^2 \in L [\hat{\omega}^2]$, we can consider $\tilde{\omega} = L^{-2} \omega$, which now solves
\[
\tilde{\omega}^2 \wedge F_{H} = 0, \quad d (| \Omega|_{\tilde{\omega}} \, \tilde{\omega}{}^2)=0,
\]
\[
\ddb \tilde{\omega} = \alpha' ({\rm Tr} \, R_{\tilde{\omega}} \wedge R_{\tilde{\omega}} - {\rm Tr} \, F_{H} \wedge F_{H}),
\]
for $\alpha' = L^{-2}$, and with
\[
[  | \Omega|_{\tilde{\omega}} \, \tilde{\omega}^2] = [\hat{\omega}^2] \in H^{2,2}_{BC}(X,\mathbb{R}).
\]
From this point of view, the scale of the balanced class is fixed, but the parameter $\alpha'$ is tunable and we obtain a sequence of solutions as $\alpha' \rightarrow 0$.

M. Garcia-Fernandez and R. Gonzalez Molina have recently discovered \cite{GM} a Futaki invariant, which on $\ddb$-manifolds is a pairing $\langle \F_0, \tau \rangle \in \mathbb{C}$ for each balanced class $\tau \in H^{2,2}_{BC}(\mathbb{R})$. They proved that if $\tau$ admits a solution to the system (more precisely, the version of the system described in \S \ref{section:tangentinstantons}), then $\langle \F_0,\tau \rangle = 0$. As a consequence of \S \ref{section:tangentinstantons}, on a simply-connected K\"ahler Calabi-Yau threefold then $\langle \F_0, \tau \rangle =0$ for all balanced classes of the form $\tau = [\omega^2]$ where $\omega$ is a K\"ahler metric.

The cone of balanced classes on a K\"ahler manifold was studied by Fu-Xiao \cite{FuXiao}, and they found that not every balanced class comes from the square of a K\"ahler metric (see \cite{FuXiao,Tosatti} for examples). Thus to complete the picture, it remains to understand balanced classes which do not come from K\"ahler classes, but rather are squares of classes on the boundary of the K\"ahler cone. 

\bigskip
\par {\bf Acknowledgements:} We thank M. Garcia-Fernandez and R. Gonzalez Molina for comments and discussions.

\section{The linearized system}
\subsection{Setup}
Let $X$ be a compact K\"ahler manifold of complex dimension $n=3$ admitting a holomorphic volume form $\Omega$. To solve the system, we will deform a pair of metrics $(\hat{H},\hat{\omega})$ where $\hat{\omega}$ is a K\"ahler Ricci-flat on $X$ and $\hat{H}$ is a vector bundle metric on $E \rightarrow X$ solving the Hermitian-Yang-Mills equation with respect to $\hat{\omega}$.
\smallskip
\par $\bullet$ To deform $\hat{H}$, we will look for solutions of the form $e^u \hat{H}$ with
\[
u \in H_0(E) = \bigg\{  u \in \Gamma({\rm End} \, E) : u^{\dagger_{\hat{H}}} = u, \ {\rm Tr} \, u = 0 \bigg\}.
\]
 Our conventions for the bundle metric $\hat{H}$ are as follows. First, the inner product on sections $u,v \in \Gamma(E)$ is given in a local trivialization by
\[
\langle u,v \rangle_{\hat{H}} = u^\alpha \hat{H}_{\alpha \bar{\beta}} \overline{v^\beta},
\]
and the adjoint $u^{\dagger_{\hat{H}}}$ of an endomorphism $u$ is with respect to this inner product. Next, we denote the inverse of the local matrix $\hat{H}$ by $\hat{H}^{\bar{\beta} \alpha}$, and if $H = e^u \hat{H}$ then $(e^u)_\alpha{}^\beta = H_{\alpha\bar{\mu}} \hat{H}^{\bar{\mu} \beta}$. For $u \in H_1(E)$, we write $F_u = F_{H}= \bar{\partial} ((\partial H) H^{-1})$ to denote the curvature of the Chern connection of $H=e^u \hat{H}$. Given a hermitian metric $\omega= i g_{j \bar{k}} dz^j \wedge d \bar{z}^k$ on the complex manifold, we use the notation
\be \label{LambdaF}
i\Lambda_\omega F_H = - g^{j \bar{k}} \partial_{\bar{k}} ( \partial_j H H^{-1}) \in \Gamma({\rm End} \, E)
\ee
so that
\[
\omega^2 \wedge i F_H = (i \Lambda_\omega F_H) \, \omega^3,
\]
and the Hermitian-Yang-Mills equation is $i \Lambda_\omega F_u = 0$. We note that
\be \label{traceLambdaF}
{\rm Tr} \,  i F_u = -  i \partial \bar{\partial} \log \det e^u + {\rm Tr} \, i F_{\hat{H}}.
\ee
By the Donaldson-Uhlenbeck-Yau theorem \cite{Donaldson,UY}, if $E \rightarrow (X,\hat{\omega})$ is stable then it admits a metric solving the Hermitian-Yang-Mills equation. Recall that $E \rightarrow (X,\hat{\omega})$ is stable if the following stability condition holds: for all coherent subsheaves $S \subset E$ with $0< {\rm rk}(S) < {\rm rk}(E)$, then
\[
\mu(S) < \mu(E),
\]
where $\mu(S) = {1 \over {\rm rk}(S)} c_1(S) \wedge [\hat{\omega}^2]$. 
\smallskip
\par $\bullet$ To deform $\hat{\omega}$, we will look for solutions of the form
\be \label{balanced-class}
 |\Omega |_{\omega_\Theta} \omega_{\Theta}^2 := |\Omega|_{\hat{\omega}} \, \hat{\omega}^2 + \Theta > 0
\ee
where
\[
\Theta \in \mathcal{U} = \bigg\{ \Theta \in {\rm Im} \, \ddb \cap \Omega^{2,2}(X): \quad  |\Omega |_{\hat{\omega}} \, \hat{\omega}^2 + \Theta > 0 \bigg\}.
\]
Since $\hat{\omega}$ is K\"ahler and $|\Omega|_{\hat{\omega}}$ is constant, then $d (|\Omega|_{\hat{\omega}} \hat{\omega}^2)=0$. For an arbitrary $\Theta \in \mathcal{U}$, the formula (\ref{balanced-class}) defines a conformally balanced metric $\omega_\Theta$ \cite{Michelsohn} by taking a square root of a positive $(2,2)$ form. To be concrete, we note the explicit formula which can be found in \cite{PPZ18a}. If we write
\[
\Psi =  \sum_{k,j}c_{kj} \Psi^{k \bar{j}} dz^1 \wedge d \bar{z}^1 \wedge \dots \wedge \widehat{dz^k} \wedge d \bar{z}^k \wedge \dots \wedge dz^j \wedge \widehat{d \bar{z}^j} \wedge \dots \wedge dz^3 \wedge d \bar{z}^3,
\]
with $c_{kj} = i^{2} 2!  (sgn(k,j)) $, then the metric $\omega = i g_{j \bar{k}} dz^j \wedge d \bar{z}^k$ solving the equation $| \Omega |_{\omega} \, \omega^{2} = \Psi$ is locally given by
\be \label{square-root-formula}
g_{j \bar{k}} = {\det (\Psi^{p \bar{q}}) \over f(z) \overline{f(z)}} (\Psi^{-1})_{j \bar{k}},
\ee
where $\Omega = f(z) dz^1 \wedge dz^2 \wedge dz^3$ in local coordinates,
\be \label{dilaton-defn}
|\Omega |^2_\omega = {f(z) \overline{f(z)} \over \det g_{j \bar{k}}},
\ee
and $\Psi^{k \bar{j}} (\Psi^{-1})_{\ell \bar{j}} = \delta^k{}_\ell$.
\smallskip
\par $\bullet$ Our space of deformations of $(\hat{H},\hat{\omega})$ is then
\[
Z =  H_0(E) \times \mathcal{U}
\]
and we let $\F: \mathbb{R} \times Z \rightarrow W$ be given by
\be \label{defn-F}
\F \left( \alpha', (u,\Theta) \right) = \begin{bmatrix}  |\Omega|_{\omega_\Theta} e^{-u/2} [\omega_\Theta^2 \wedge  i F_{u}] e^{u/2} \\ \ddb \omega_\Theta - \alpha' ({\rm Tr} \, R_\Theta \wedge R_\Theta - {\rm Tr} \, F_{u} \wedge F_{u}) \end{bmatrix}.
\ee
The image of $\F$ is contained in
\[
  W = V \times ({\rm Im} \, \ddb \cap \Omega^{2,2}(X)),
\]
where
\[
V= \bigg\{  s \in \Omega^6(X,{\rm End} \, E) : s^{\dagger_{\hat{H}}} = s, \ \int_X {\rm Tr} \, s = 0 \bigg\}.
\]
We now verify that the image $\F$ is indeed in $W$. First, we note that
\[
\int_X |\Omega|_{\omega_\Theta} \omega_\Theta^2 \wedge {\rm Tr} \, [e^{-u/2} i F_{u} e^{u/2}] = \int_X |\Omega|_{\hat{\omega}} \hat{\omega}^2 \wedge {\rm Tr} \, i F_{\hat{H}} = 0
\]
by (\ref{traceLambdaF}), (\ref{balanced-class}), and Stokes's theorem. Next, we verify the condition
\[
\bigg[ e^{-u/2} (i\Lambda_{\omega_\Theta} F_{u}) e^{u/2} \bigg]^{\dagger_{\hat{H}}} = e^{-u/2} (i\Lambda_{\omega_\Theta} F_{u}) e^{u/2}.
\]
The expression (\ref{LambdaF}) shows that $(i \Lambda_\omega F) H$ is
a hermitian matrix, which we write $[i \Lambda_\omega F]^{\dagger_H}=i
\Lambda_\omega F$ i.e. $(i \Lambda F)$ is selfadjoint with respect to
$H$. We will need the image of $\F$ to be selfadjoint with respect to
the reference $\hat{H}$, which is the reason for conjugating by
$e^{u/2}$. In general, given two hermitian positive definite matrices
$\hat{H}$, $H$ and and a matrix $A$, if $AH$ is a hermitian matrix then $(h^{-1/2} A h^{1/2}) \hat{H}$ is a hermitian matrix where $h = H \hat{H}^{-1}$. This shows that $e^{-u/2} (i\Lambda_{\omega_\Theta} F_{u}) e^{u/2}$ is self-adjoint with respect to $\hat{H}$.

\par Finally, the hypothesis $c_2(T^{1,0}X)=c_2(E)$ guarantees that the image of the second component of $\F$ is $\ddb$ exact by the $\ddb$-lemma.

\par We note that $\F(0,(0,0))=0$. We will use the implicit function theorem to show that for each $\alpha' \in (-\epsilon,\epsilon)$ there is a pair $(u_\alpha,\Theta_\alpha)$ so that $\F(\alpha',(u_\alpha,\Theta_\alpha)) = 0$. For this, we will calculate the linearization of $\F$ at the origin in the $W$ directions, which we write as
\[
(D_2 \F)|_{0}(\dot{u}, \dot{\Theta}) = \begin{bmatrix} L_1 & A  \\ 0 & L_2 \end{bmatrix} \begin{bmatrix} \dot{u} \\ \dot{\Theta} \end{bmatrix}.
\]
We now compute $L_1,A,L_2$.

\subsection{Linearized Yang-Mills equation}
Let $H(t)$ be a path of metrics with $H(0) = \hat{H}$. Differentiating the curvature gives
\[
{d \over dt} \bigg|_{t=0} F_{j \bar{k}} = - \partial_{\bar{k}} {d \over dt} \bigg|_{t=0} (\partial_j H H^{-1})
\]
which can be written as
\[
{d \over dt} \bigg|_{t=0} F_{j \bar{k}} = - \partial_{\bar{k}} \nabla_j ( \dot{h} h^{-1}),
\]
with $h = H \hat{H}^{-1}$ and $\nabla$ the Chern connection of $\hat{H}$ acting on sections of ${\rm End} \, E$. Therefore, along a path $H= e^{u(t)} \hat{H}$ with $u(0)=0$ and $F_{\hat{H}} \wedge \hat{\omega}^2 = 0$, then
\[
{d \over dt} \bigg|_{t=0} \bigg[  |\Omega|_{\hat{\omega}} e^{-u(t)/2} [\hat{\omega}^2 \wedge  i F_{u(t)}] e^{u(t)/2} \bigg] =  - \hat{g}^{j \bar{k}} \partial_{\bar{k}} \nabla_j \dot{u} \, \otimes |\Omega|_{\hat{\omega}} {\hat{\omega}^3 \over 3!}.
\]
We let $L_1: \Gamma({\rm End} \, E) \rightarrow \Omega^6({\rm End \, E})$ be
\[
L_1 h := - \hat{g}^{j \bar{k}} \partial_{\bar{k}} \nabla_j h \, \otimes |\Omega|_{\hat{\omega}} {\hat{\omega}^3 \over 3!}.
\]
We can also vary the first component of $\F$ along a path of hermitian metrics on the base manifold $X$. Let $\Theta(t)$ be a path of exact $(2,2)$-forms with $\Theta(0)=0$. Then
\[
{d \over dt} \bigg|_{t=0} \bigg[  |\Omega|_{\omega_{\Theta(t)}} \, \omega_{\Theta(t)}^2 \wedge  i \hat{F}  \bigg] =  2 |\Omega|_{\hat{\omega}} \, \dot{\omega} \wedge \hat{\omega} \wedge i \hat{F}
\]
where $\dot{\omega} = {d \over dt} \big|_{t=0} \, \omega_{\Theta(t)}$ will be computed in the next section. The result will be (\ref{variation-ddb-omega}), and we let
\[
A(\dot{\Theta}) := \Lambda_{\hat{\omega}} \dot{\Theta} \wedge \hat{\omega} \wedge i \hat{F}.
\]
To end this section, we note the identity
\be \label{bochnerformula}
\int_X \langle L_1 h, h \rangle_{\hat{H}} =  {1 \over 2} \int_X \bigg[ |\partial^{\hat{H}} h|^2_{\hat{g},\hat{H}} + |\bar{\partial} h|^2_{\hat{g},\hat{H}} \bigg] |\Omega|_{\hat{\omega}} {\hat{\omega}^3 \over 3!},
\ee
which holds for all self-adjoint endomorphisms $h \in \Gamma({\rm End} \, E)$. The inner product on endomorphisms is $\langle h_1, h_2 \rangle_{\hat{H}} = {\rm Tr} \, h_1 h_2^{\dagger_{\hat{H}}}$. To derive (\ref{bochnerformula}), we start with
\bea
\hat{g}^{j \bar{k}}  \partial_{\bar{k}} \partial_j \langle h,h \rangle_{\hat{H}} &=&  \langle \hat{g}^{j \bar{k}} \partial_{\bar{k}} \nabla_j h,h \rangle_{\hat{H}} +  \langle h, \hat{g}^{\bar{j} k} \nabla_k \partial_{\bar{j}} h  \rangle_{\hat{H}}  +  \hat{g}^{j \bar{k}} \langle \nabla_j h, \nabla_k h \rangle_{\hat{H}} \nonumber\\
&&+  \hat{g}^{j \bar{k}} \langle \partial_{\bar{k}} h, \partial_{\bar{j}} h \rangle_{\hat{H}}.
\eea
Since $\hat{g}^{j \bar{k}} \hat{F}_{j \bar{k}} =0$, we may freely commute covariant derivatives so that $\hat{g}^{j \bar{k}} \nabla_k \partial_{\bar{j}} h = \hat{g}^{j \bar{k}} \partial_{\bar{j}} \nabla_k h$. Then
\[
\hat{g}^{j \bar{k}}  \partial_{\bar{k}} \partial_j \langle h,h \rangle_{\hat{H}} = 2 \langle  \hat{g}^{j \bar{k}} \partial_{\bar{k}} \nabla_j h, h \rangle_{\hat{H}} + |\partial^{\hat{H}} h|^2_{\hat{g},\hat{H}} + |\bar{\partial} h|^2_{\hat{g},\hat{H}}.
\]
Multiplying through by $|\Omega|_{\hat{\omega}} {\hat{\omega}^3 \over 3!}$, integrating and applying Stokes's theorem gives (\ref{bochnerformula}).

\subsection{Linearized conformally balanced equation}

We now record how $\omega_\Theta$ varies under a change in $\Theta$. This variational formula is fundamental in the study of the Anomaly flow, and its derivation can be found in \cite{PPZ18b}. For other related setups where a similar formula is used, see \cite{BedVez,FeiPhong,FeiPic,PPZ19}. We reproduce the full calculation here for the sake of completeness.

\begin{lem} \cite{PPZ18b} Let $\omega>0$ solve $| \Omega |_\omega \omega^2 = \Theta$. Then
\be \label{variation-g}
(\delta g)_{j \bar{k}} = -{1 \over 2 | \Omega |_\omega} g^{s \bar{r}} (\delta \Theta)_{s \bar{r} j \bar{k}}.
\ee
If we use the notation
\[
(i\Lambda_\omega \Psi)_{j \bar{k}} =  g^{a \bar{b}} \Psi_{a \bar{b} j \bar{k}}
\]
for $\Psi \in \Omega^{2,2}(X)$, then
\be \label{variation-omega}
\delta \omega = {1 \over 2 | \Omega |_\omega} \Lambda_\omega \delta \Theta.
\ee
\end{lem}

\begin{proof} Taking the variation of $| \Omega |_\omega \omega^2 = \Theta$ gives
\[
(\delta | \Omega |_{\omega} ) \, \omega^2 + 2 | \Omega |_{\omega} \delta \omega \wedge \omega = \delta \Theta.
\]
Differentiating the definition of $|\Omega|_\omega$ (\ref{dilaton-defn}), we obtain
\[
\delta | \Omega |_\omega = -{1 \over 2} | \Omega |_\omega \, g^{j \bar{k}} (\delta g)_{j \bar{k}}.
\]
Therefore
\be \label{linearized-conf-bal}
\bigg[ - {1 \over 2}  g^{j \bar{k}} (\delta g)_{j \bar{k}} \bigg] \omega^2 + 2 \delta \omega \wedge \omega = {1 \over | \Omega |_\omega} \delta \Theta.
\ee
The computation to extract $\delta g$ was done in \cite{PPZ18b}. One approach is to use the Hodge star operator, and another is to expand both sides in components. We will take the component approach. Using the convention
\[
\Psi = {1 \over 4} \Psi_{s \bar{r} j \bar{k}} \, dz^s \wedge d \bar{z}^r \wedge dz^j \wedge d \bar{z}^k,
\]
and $\omega = i g_{j \bar{k}} dz^j \wedge d \bar{z}^k$, we have
\[
\omega^2 = (i g_{s \bar{r}} dz^s \wedge d \bar{z}^r) \wedge (i g_{j \bar{k}} dz^j \wedge d \bar{z}^k)
\]
which is
\[
\omega^2 = (- g_{s \bar{r}}g_{j \bar{k}}) \,  d z^s \wedge d \bar{z}^r \wedge dz^j \wedge d \bar{z}^k 
\]
and after antisymmetrization becomes
\be \label{omega^2}
(\omega^2)_{s \bar{r} j \bar{k}} = - 2 g_{s \bar{r}}g_{j \bar{k}} + 2 g_{j \bar{r}} g_{s \bar{k}}.
\ee
Next, we have
\[
\omega \wedge \delta \omega = (i g_{s \bar{r}} dz^s \wedge d \bar{z}^r) \wedge (i \delta g_{j \bar{k} } \, dz^j \wedge d \bar{z}^k),
\]
which becomes
\be \label{delta-omega}
(\omega \wedge \delta \omega)_{s \bar{r} j \bar{k}} = - g_{s \bar{r}} (\delta g)_{j \bar{k}} +  g_{j \bar{r}} (\delta g)_{s \bar{k}} - g_{j \bar{k}} (\delta g)_{s \bar{r}} +g_{s \bar{k}} (\delta g)_{j \bar{r}}.
\ee
Substituting (\ref{omega^2}) and (\ref{delta-omega}) into (\ref{linearized-conf-bal}), we obtain
\bea
& \ & \bigg[ - {1 \over 2}  g^{p \bar{q}} (\delta g)_{p \bar{q}} \bigg] \bigg[ - 2 g_{s \bar{r}}g_{j \bar{k}} + 2 g_{j \bar{r}} g_{s \bar{k}} \bigg] \nonumber\\
&&+ 2 \bigg[ - g_{s \bar{r}} (\delta g)_{j \bar{k}} +  g_{j \bar{r}} (\delta g)_{s \bar{k}} - g_{j \bar{k}} (\delta g)_{s \bar{r}} +g_{s \bar{k} } (\delta g)_{j \bar{r}} \bigg] \nonumber\\
&=& {1 \over | \Omega |_\omega} (\delta \Theta)_{s \bar{r} j \bar{k}}.
\eea
Contracting by $g^{s \bar{r}}$, we obtain
\bea
& \ & \bigg[ - {1 \over 2}  g^{s \bar{r}} (\delta g)_{s \bar{r}} \bigg] \bigg[- 2 (3) g_{j \bar{k}} + 2 g_{j \bar{k}} \bigg] \nonumber\\
&&+ 2 \bigg[ - 3 (\delta g)_{j \bar{k}} +  (\delta g)_{j \bar{k}} - [ g^{s \bar{r}} (\delta g)_{s \bar{r}}] g_{j \bar{k}} + (\delta g)_{j \bar{k}} \bigg] \nonumber\\
&=& {1 \over | \Omega |_\omega} g^{s \bar{r}} (\delta \Theta)_{s \bar{r} j \bar{k}}.
\eea
The terms $g^{s \bar{r}} (\delta g)_{s \bar{r}}$ cancel exactly. Simplifying gives (\ref{variation-g}), which proves the lemma.
\end{proof}

We will need the variation of $\ddb \omega_\Theta$. Let $\Theta(t)$ be a 1-parameter family of closed $(2,2)$ forms with $\Theta(0)=0$. Then by (\ref{variation-omega}), we have
\be \label{variation-ddb-omega}
{d \over dt} \bigg|_{t=0} \ddb \omega_{\Theta(t)} = \ddb \bigg[ {1 \over 2 | \Omega |_{\hat{\omega}}} \Lambda_{\hat{\omega}} \dot{\Theta}\bigg],
\ee
where $\omega_{\Theta(0)} = \hat{\omega}$ and $\dot{\Theta} = {d \over dt} |_{t=0} \Theta$. If $\hat{\omega}$ is K\"ahler Ricci-flat, then $| \Omega |_{\hat{\omega}}$ is constant and the K\"ahler identity $[\Lambda, \bar{\partial}] = - i \partial^\dagger$ implies
\[
{d \over dt} \bigg|_{t=0} \ddb \omega_{\Theta_t} =  - {1 \over 2 | \Omega |_{\hat{\omega}}} \partial \partial^\dagger_{\hat{\omega}} \dot{\Theta} = - {1 \over 2 | \Omega |_{\hat{\omega}}} \Delta_{\hat{\omega}} \dot{\Theta},
\]
using that $d \dot{\Theta} = 0$. Here $\Delta_\omega = \partial \partial^\dagger_{\omega} + \partial^\dagger_{\omega} \partial$. We then denote
\[
L_2 \dot{\Theta} := - {1 \over 2 | \Omega |_{\hat{\omega}}} \Delta_{\hat{\omega}} \dot{\Theta},
\]
to be the linearization of $\ddb \omega_\Theta$ along paths of metrics constrained to the form $|\Omega|_{\omega_\Theta} \omega_\Theta^2 = |\Omega|_{\hat{\omega}} \hat{\omega}^2 + \Theta$.

\subsection{Implicit function theorem}
Let $\F : \mathbb{R} \times C^{k+2,\alpha}(Z) \rightarrow C^{k,\alpha}(W)$, denoted $\F(\alpha',(u,\Theta))$, be given by (\ref{defn-F}) as before. Combining our work so far, we computed the linearization of $\F$ at the origin in the $W$ directions to be
\[
  D_2 \F|_0 : C^{k+2,\alpha} \bigg (H_0(E) \times ({\rm Im} \, \ddb \cap \Omega^{2,2}) \bigg) \rightarrow C^{k,\alpha} \bigg(V \times  ({\rm Im}  \, \ddb \cap \Omega^{2,2}) \bigg)
\]
\[
  (D_2 \F)|_{0}(\dot{u}, \dot{\Theta}) = \begin{bmatrix} L_1 & A  \\ 0 & L_2 \end{bmatrix} \begin{bmatrix} \dot{u} \\ \dot{\Theta} \end{bmatrix},
\]
where the diagonal operators are
\bea
L_1 &:&  C^{k+2,\alpha}(H_0(E))\rightarrow C^{k,\alpha}(V) \nonumber\\
L_1 &=& -\hat{g}^{j \bar{k}} \nabla_{\bar{k}} \nabla_j \otimes |\Omega|_{\hat{\omega}} {\hat{\omega}^3 \over 3!} \nonumber
\eea
and
\bea
L_2 &:&  C^{k+2,\alpha}(({\rm Im} \, \ddb \cap \Omega^{2,2})) \rightarrow C^{k,\alpha}(({\rm Im} \, \ddb \cap \Omega^{2,2})) \nonumber\\
L_2 &=& - {1 \over 2 | \Omega |_{\hat{\omega}}} \Delta_{\hat{\omega}} \nonumber
\eea
and the off-diagonal is
\bea
A &:&  C^{k+2,\alpha}(({\rm Im} \, \ddb \cap \Omega^{2,2})) \rightarrow C^{k,\alpha}(V) \nonumber\\
A(\dot{\Theta}) &=& \Lambda_{\hat{\omega}} \dot{\Theta} \wedge \hat{\omega} \wedge i F_{\hat{H}} .\nonumber
\eea
The diagonal operators $L_1,L_2$ are both invertible and thus $D_2 \F|_0$ is invertible. Indeed:

\par $\bullet$ Invertibility of $L_2$: the operator $L_2$ is invertible by Hodge theory and the $\ddb$-lemma. Said otherwise, on a K\"ahler manifold we can view the domain and range as $L_2: {\rm Im} \, d \rightarrow {\rm Im} \, d$ and exchange $\partial \partial^\dagger+\partial^\dagger \partial$ with the usual Hodge Laplacian $dd^\dagger+d^\dagger d$.

\par $\bullet$ Invertibility of $L_1$: the operator $L_1$ has an adjoint $L_1^\dagger: \Omega^6({\rm End} \, E) \rightarrow \Gamma({\rm End} \, E)$ with $L^2$ inner product
\[
  ( u \otimes {\hat{\omega}^3 \over 3!}, v \otimes {\hat{\omega}^3 \over 3!})_{L^2} = \int_X \langle u,v \rangle_{\hat{H}} {\hat{\omega}^3 \over 3!}.
\]
By (\ref{bochnerformula}), we know that if
\be \label{L2-zero}
(L_1 u, u \otimes  {\hat{\omega}^3 \over 3!})_{L^2} = 0,
\ee
then $\bar{\partial} u = 0$. Therefore if $u \in {\rm Ker} \, L_1$, then $u$ is a holomorphic endomorphism of $E$. Since $E \rightarrow X$ is stable, $u$ must be a multiple of the identity, otherwise considering the subsheaves ${\rm Ker} \, u \subset E$, ${\rm Im} \, u \subset E$ violates the stability condition.
\[
{\rm Ker} \, L_1 = \mathbb{C} \cdot {\rm id}.
\]
Since $L_1: H_0(E) \rightarrow V$, then ${\rm Ker} \, L_1 = \{ 0 \}$. Similarly, if $u \otimes  {\hat{\omega}^3 \over 3!} \in {\rm Ker} \, L_1^\dagger$, then (\ref{L2-zero}) holds and
\[
{\rm Ker} \, L_1^\dagger = \mathbb{C} \cdot {\rm id} \otimes {\hat{\omega}^3 \over 3!}
\]
which implies ${\rm Im} \, L_1 =({\rm Ker} \, L_1^\dagger)^\perp = V$.

\par $\bullet$ Invertibility of $D_2 \F|_0$: this follows since the diagonal operators are invertible. 
\[
(D_2 \F|_0)^{-1} = \begin{bmatrix} L_1^{-1} & - L_1^{-1} A L_2^{-1} \\ 0 & L_2^{-1} \end{bmatrix}.
\]
By the implicit function theorem, there exists $\epsilon>0$ such that for all $\alpha' \in (-\epsilon,\epsilon)$, there exists $(u,\Theta) \in C^{k+2,\alpha}(Z)$ such that $\F(\alpha',(u,\Theta))=0$.

\begin{rk} From a solution $\F(\alpha',(u,\Theta))=0$ with small $\alpha'>0$, we can consider $\tilde{\omega} = \alpha'^{-1} \omega_\Theta$, which solves
\be \label{anomaly-alpha1}
\ddb \tilde{\omega} = {\rm Tr} \, R_{\tilde{\omega}} \wedge R_{\tilde{\omega}} - {\rm Tr} \, F_{H_u} \wedge F_{H_u},
\ee
but changes the balanced class to
\[
 | \Omega |_{\tilde{\omega}} \tilde{\omega}^2 \in \alpha'^{-1/2}  [|\Omega|_{\hat{\omega}} \hat{\omega}^2].
\]
As $\alpha' \rightarrow 0$, we obtain a sequence of solutions to (\ref{anomaly-alpha1}) in balanced classes of the form $M_0 [\hat{\omega}^2]$ with radius $M_0=\alpha'^{-1/2} |\Omega|_{\hat{\omega}}$ tending to infinity $M_0 \rightarrow \infty$.
\end{rk}

\section{Further remarks}
\subsection{Ellipticity and regularity}
In the previous section, we use the implicit function theorem to obtain $C^{3,\alpha}$ solutions to $\F(\alpha',(u,\Theta))=0$. We will now show that these $C^{3,\alpha}$ solutions are smooth for $\alpha'$ small enough. This will follow from ellipticity of the system of equations in the regime $|\alpha' R_g|_g \ll 1$. Note that since the anomaly cancellation equation is a fully nonlinear PDE system for the metric due to the ${\rm Tr} \, R_g \wedge R_g$ term, we do not expect the equation to be elliptic everywhere, but it should be elliptic in an open set of solution space.

\par The solutions $(\alpha', (u_\alpha,\Theta_\alpha))$ obtained by the implicit function theorem lie in the regime $|\alpha' R_g|_g \ll 1$, since as $\alpha' \rightarrow 0$ then $\Theta_\alpha \rightarrow 0$, $R_{\Theta_\alpha} \rightarrow R_{\hat{g}}$ and so $|\alpha' R_{\Theta_\alpha}| \rightarrow 0$.

\begin{prop}
Consider a solution $(H,g_\Theta) \in C^{3,\alpha}$ to the equations
\bea
(g_\Theta)^{j \bar{k}} (F_H)_{j \bar{k}} &=& 0, \nonumber\\
\ddb \omega_{\Theta} - \alpha' ({\rm Tr} \, R_\Theta \wedge R_\Theta - {\rm Tr} \, F_H \wedge F_H) &=& 0 \nonumber
\eea
with $|\Omega|_{\omega_\Theta} \omega_\Theta^2 = |\Omega|_{\hat{\omega}} \, \hat{\omega}^2 + \Theta$ and $\Theta = \ddb \beta$. Suppose $g_\Theta$ lies in the ellipticity region $|\alpha' R_{g_\Theta}|_{g_{\Theta}} < \epsilon$ with $\epsilon>0$ a universal constant. Then the pair $(H,g_\Theta)$ is smooth.
\end{prop}
\begin{proof} We work in a local coordinate chart $B_1$. Here the first equation is of the form
\[
g_\Theta^{j \bar{k}} \partial_j \partial_{\bar{k}} H = \mathcal{O}(H,g_\Theta,\partial H) \in C^{2,\alpha}(B_1).
\]
By interior Schauder estimates, we obtain that the components $H_{\alpha \bar{\beta}} \in C^{4,\alpha}(B_{1/2})$. Next, we upgrade the regularity of the metric $g_\Theta$.  Let $\delta$ denote differentiation of local components in a coordinate direction, e.g. $\delta = {\partial \over \partial x^k}$. Differentiating the anomaly equation gives
\[
\ddb \delta \omega_\Theta - 2 \alpha' {\rm Tr} \, R_\Theta \wedge \delta R_\Theta + 2 \alpha' {\rm Tr} \, F_H \wedge \delta F_H = 0.
\]
We computed $\delta \omega_\Theta$ earlier (\ref{variation-omega}), and the variation of the curvature of the Chern connection is
\[
\delta R_{g_\Theta} = \bar{\partial} \partial_\nabla (\delta g_\Theta \, g_\Theta^{-1} ).
\]
The differentiated equation becomes
\bea
& \ & \ddb \bigg[ {1 \over 2 |\Omega|_{\omega_\Theta}} \Lambda_{\omega_\Theta} \delta \Theta \bigg] \\
&=& \ \alpha' \bigg[ {\rm Tr} \, R_\Theta \wedge \bar{\partial} \partial_\nabla  \left( {1 \over  |\Omega|_{\omega_\Theta}}  (\Lambda_{\omega_\Theta} \delta \Theta) \, g_\Theta^{-1} \right) - 2 {\rm Tr} \, F_H \wedge \bar{\partial} \partial_\nabla (\delta H \, H^{-1}) \bigg] \nonumber
\eea
We note the following non-K\"ahler identity, which can be found in Demailly \cite{bigDemailly} Chapter VI Theorem 6.8:
\[
[\Lambda,\bar{\partial}] = - i \partial^\dagger - i [\Lambda,\partial \omega]^\dagger.
\]
Since $\Theta = \ddb \beta$, then $\delta \Theta$ is closed and
\[
\ddb \bigg[ {1 \over 2 |\Omega|_{\omega_\Theta}} \Lambda_{\omega_\Theta} \delta \Theta \bigg] = -{1 \over 2 |\Omega|_{\omega_\Theta}} \Delta_{\partial} \delta \Theta + \mathcal{O}(\Theta,D \Theta, D^2 \Theta).
\]
It follows that
\[
\ddb \bigg[ {1 \over 2 |\Omega|_{\omega_\Theta}} \Lambda_{\omega_\Theta} \delta \Theta \bigg] = {1 \over 2 |\Omega|_{\omega_\Theta}} (g_\Theta)^{j \bar{k}} \partial_j \partial_{\bar{k}} \delta \Theta + \mathcal{O}(\Theta,D \Theta, D^2 \Theta).
\]
The local equation on $B_1$ can then be written as
\[
(g_\Theta)^{j \bar{k}} \partial_j \partial_{\bar{k}} \delta \Theta  - 2 \alpha' {\rm Tr} \, R_\Theta \wedge  (\Lambda_{\omega_\Theta} \bar{\partial} \partial_\nabla \delta \Theta) g_\Theta^{-1} \in C^{1,\alpha} 
\]
which is of the form
\[
A^{IJ p \bar{q}} \partial_p \partial_{\bar{q}} (\delta \Theta)_J = f^I
\]
with $f \in C^{1,\alpha}$, $A^{IJ p \bar{q}} \in C^{1,\alpha}$ and
\[
A^{IJ p \bar{q}} \xi_p \overline{\xi_q} \tau_I \tau_J \geq \lambda |\xi|^2 |\tau|^2
\]
for $2 \alpha' |R_\Theta| \ll 1$ small enough. By the Schauder estimates for elliptic systems (e.g. \cite{Giaquinta}), we conclude that $\delta \Theta \in C^{3,\alpha}$. Therefore $\Theta \in C^{4,\alpha}$, and the solution $(g_\Theta,H) \in C^{4,\alpha}$. Repeating this process gives $(g_\Theta,H) \in C^{k,\alpha}$ for every positive integer $k$.
\end{proof}

\subsection{Instantons on the tangent bundle} \label{section:tangentinstantons}
There are other versions of the heterotic system considered in the literature which combine observations of Strominger \cite{Strom}, Hull \cite{Hull}, and Ivanov \cite{Ivanov}, and these involve an instanton connection on the tangent bundle. For recent work using this setup, see e.g. \cite{ASTW,dlOSvanes,GRT,GRST,GM}.

We now describe the equations solved by Andreas and Garcia-Fernandez \cite{AGF} by implicit function theorem on K\"ahler manifolds. The system of equations in this case is for a triple $(\omega,\nabla_1,\nabla_2)$, where $g$ is a metric on $T^{1,0}X$, $\nabla_1$ is a connection on $T^{1,0}X$ which is unitary with respect to $g$, and $\nabla_2$ is a connection on $E$, solving
\be \label{setup1}
R_{\nabla_1}^{0,2}=R_{\nabla_1}^{2,0}=0, \quad F_{\nabla_2}^{0,2}=F_{\nabla_2}^{2,0}=0,
\ee
\[
\omega^2 \wedge R_{\nabla_1} = 0, \quad  \omega^2 \wedge F_{\nabla_2} = 0,
\]
  \[
d (| \Omega|_\omega \, \omega^2)=0,
\]
\[
\ddb \omega = \alpha' ({\rm Tr} \, R_{\nabla_1} \wedge R_{\nabla_1} - {\rm Tr} \, F_{\nabla_2} \wedge F_{\nabla_2}),
\]
where $\omega = i g_{j \bar{k}} dz^j \wedge d \bar{z}^k$. Assuming that the holomorphic tangent bundle is stable (by \cite{CPY,FLY} this is true for non-K\"ahler threefolds created by a conifold transition), a way to find solutions to this system is to instead fix a holomorphic structure on $(E, T^{1,0}X)$, and look for a triple $(g,h,H)$ where $g,h$ are metrics on $T^{1,0}X$ and $H$ is a metric on $E$ solving
\be \label{setup2}
\omega^2 \wedge R_h = 0, \quad \omega^2 \wedge F_{H} = 0, \quad d (| \Omega|_\omega \, \omega^2)=0,
\ee
\[
\ddb \omega = \alpha' ({\rm Tr} \, R_h \wedge R_h - {\rm Tr} \, F_{H} \wedge F_{H}),
\]
where $R_h,F_H$ are the curvatures of the Chern connections of the metrics $h,H$. Denote the Chern connection of $h$ by $\nabla^h$. To use solutions of (\ref{setup2}) to solve the equations from setup (\ref{setup1}), we need to change $\nabla^h$ to be unitary with respect to $g$. Our conventions are $\langle u,v \rangle_h = h_{j \bar{k}} u^j \overline{v^k}$, which we write in matrix notation as $\langle u,v \rangle_h = u^T h \bar{v}$. Define the gauge transformation $\sigma$ by $g = \bar{\sigma}^\dagger h \bar{\sigma}$. Then $\tilde{\nabla} = \sigma^{-1} \circ \nabla^h \circ \sigma$ is unitary with respect to $g$, as direct computation shows that
\[
\partial_i \langle u,v \rangle_g = \langle \tilde{\nabla}_i u, v \rangle_g + \langle u, \tilde{\nabla}_{\bar{i}} v \rangle_g.
\]
Since $R_{\tilde{\nabla}} = \sigma^{-1} R_{\nabla^h} \sigma$, the connection $\tilde{\nabla}$ can be used to solve the equations from setup (\ref{setup1}).

The method described here adapts to setup (\ref{setup2}) as well, and we give a sketch of the proof. Let $\hat{\omega} = i \hat{g}_{j \bar{k}} dz^j \wedge d \bar{z}^k$ be K\"ahler Ricci-flat and $\hat{H}$ solve $\hat{\omega}^2 \wedge F_{\hat{H}}=0$. We deform $(\hat{g},\hat{H},\hat{\omega})$ to $(e^{u_1} \hat{g}, e^{u_2} \hat{H}, \omega_\Theta)$ with
\[
(u_1,u_2,\Theta) \in Z =  H_0(T^{1,0}X) \times H_0(E) \times \mathcal{U}
\]
and
\[
|\Omega|_{\omega_\Theta} \omega_\Theta^2 = |\Omega|_{\hat{\omega}} \hat{\omega}^2 + \Theta
\]
as before. Let $\F: \mathbb{R} \times Z \rightarrow W$ be given by
\be 
\F \left( \alpha', (u_1,u_2,\Theta) \right) = \begin{bmatrix} |\Omega|_{\omega_\Theta} e^{-u_1/2} [\omega_\Theta^2 \wedge  i R_{u_1}] e^{u_1/2} \\  |\Omega|_{\omega_\Theta} e^{-u_2/2} [\omega_\Theta^2 \wedge i F_{u_2}] e^{u_2/2} \\ \ddb \omega_\Theta - \alpha' ({\rm Tr} \, R_{u_1} \wedge R_{u_1} - {\rm Tr} \, F_{u_2} \wedge F_{u_2}) \end{bmatrix}.
\ee
We have that $\F(0,(0,0,0))=0$, and the linearization of $\F$ at the origin is of the form
\[
  (D_2 \F)|_{0}(\dot{u}_1, \dot{u}_2, \dot{\Theta}) = \begin{bmatrix} L_1 & 0 & A_1 \\ 0 & L_2 & A_2 \\ 0 & 0 & L_3 \end{bmatrix} \begin{bmatrix} \dot{u}_1 \\ \dot{u}_2 \\ \dot{\Theta} \end{bmatrix}.
\]
Provided that $X$ is simply-connected, then $T^{1,0}X$ is stable and the methods in this paper can be applied to show that $L_1$, $L_2$, $L_3$ are invertible in suitable spaces. The implicit function theorem then gives solutions for small $\alpha' \in (-\epsilon,\epsilon)$. This shows that the result of Andreas and Garcia-Fernandez \cite{AGF} can be modified to control the balanced class of the solution as expected in \cite{GF-survey}.

\end{document}